\documentclass[]{amsart}
\usepackage{amsmath,amsthm, amssymb,xcolor}
\usepackage{graphicx}
\usepackage{graphics}

\usepackage{mathtools}

\newcommand{\R}{{\mathbb R}}

\newcommand{\sign}{{\rm sign\,}}
\newcommand{\ol} {\overline}

\newcommand{\A}{{\mathcal  A}}

\newcommand{\diam}{{\rm diam\,}}
\newcommand{\loc}{{\rm loc}}

\newcommand{\osc}{{\rm osc\,}}

\numberwithin{equation}{section}

\def\Xint#1{\mathchoice
    {\XXint\displaystyle\textstyle{#1}}%
     {\XXint\textstyle\scriptstyle{#1}}%
     {\XXint\scriptstyle\scriptscriptstyle{#1}}%
     {\XXint\scriptstyle\scriptscriptstyle{#1}}%
    \!\int}
\def\XXint#1#2#3{{\setbox0=\hbox{$#1{#2#3}{\int}$}
    \vcenter{\hbox{$#2#3$}}\kern-.5\wd0}}

\newtheorem{thm}{Theorem}[section]  %% Definition of Theorem
\newtheorem{lem}[thm]{Lemma}           %% Definition of Lemma
\newtheorem{prp}[thm]{Proposition}     %% Definition of Proposition
\newtheorem{defin}[thm]{Definition}    %% Definition of Definition

\begin{document}
%\today

%\baselineskip=20pt
\keywords{Quasilinear   parabolic equations,  Cauchy-Dirichlet problem,  $VMO_x$ coefficients,  Fixed Point Theorem,  strong solutions}

\subjclass[2020]{Primary: 35K60; Secondary:  35K20; 35R05}

\title[Quasilinear Cauchy-Dirichlet Problem]{Quasilinear Cauchy-Dirichlet problem for parabolic equations with $VMO_x$ coefficients}

\author[R. Rescigno]{Rosamaria Rescigno}

\begin{abstract}
We study the strong solvability of the Cauchy-Dirichlet problem for parabolic quasilinear 
equations with discontinuous data. The principal coefficients depend on the point $(x,t)$ and on the solution $u,$  the dependence on $x$ is of $VMO$ type while these are only measurable with respect to $t.$  Assuming suitable structural conditions on the nonlinear terms, we prove existence and uniqueness of the strong solution, which turns out to be also H\"older continuous.
\end{abstract}

\maketitle

\section{Introduction}
Let $Q=\Omega\times(0,T)$ be a cylinder in $\R^{n+1}$  where $\Omega\subset\R^n$  is a $C^{1,1}$-smooth and bounded domain, $T>0,$ and $\partial_pQ=\{\partial \Omega \times (0,T)\}\cup \{\overline{\Omega}\times\{0\}\}$ stands for the parabolic boundary of $Q.$
We consider the following Cauchy-Dirichlet problem
\begin{equation}\label{PD}
    \begin{cases}
    \mathcal{P}u := D_tu-a^{ij}(x,t,u)D_{ij}u=f(x,t,u,Du)  \qquad &\text{a.e in }Q\,,\\
  \phantom{\mathcal{P}u := }  u(x,t)=0 & \text{on } \partial_pQ\,.
    \end{cases}
\end{equation}

The classical theory concerning quasilinear elliptic and parabolic   PDEs  is well developed and can be found in various  books (cf. \cite{GT,LSU,LU}). 
After the seminal result of Chiarenza, Frasca, and  Longo \cite{CFL}, that extends  the  unique solvability   theory  for  linear elliptic PDEs  with smooth coefficients to equations with discontinuous $VMO$  coefficients,  a lot of studies appeared  dedicated to linear and quasilinear equations with such  coefficients  (cf. \cite{AFS,ANPS3,BPS,Kr1,MPS}) and nonlinear operators with coefficients close to $VMO$  (see \cite{BSf,RPS,RS} and the references therein).     Regarding the parabolic case, Krylov \cite{Kr} weakened the $VMO$ condition  on the coefficient 
$\{a^{ij}(x,t)\}$ assuming   \textit{partially} $VMO$, that means that the discontinuity  is of $VMO$-type \textit{only} with respect to  the space variables $x,$  while the coefficients are allowed to be \textit{only measurable} with respect to the time variable $t.$ Later, his results have been  used   to study  various initial-boundary value problems for linear parabolic operators with partially $VMO$ principal coefficients, see  \cite{ANPS1,ANPS2,DK,KimKr}. 

As it is known,  the linear  theory concerning existence and uniqueness of solitions to various  boundary value problems  is a starting point for the study of quasilinear and nonlinear PDEs  under different  conditions on the boundary. In this context, we can mention the work of Palagachev \cite{Pl} where quasilinear elliptic equations with $VMO$-type discontinuity  assumption on the coefficients have been studied for a first time. In that work the author evaluates the $VMO$ modulus of the composition $a^{ij}(\cdot, u(\cdot))$ that allows to apply the Leray-Schauder fixed point theorem in order to get strong solvability of  quasilinear Dirichlet problem. Similar result has been obtained  later in \cite{Sf} for  parabolic quasilinear oprerators with Dirichlet and oblique derivative conditions on the boundary.

Our goal here is to extend these  results to the case of quasilinear parabolic operators with partially $VMO$ principal coefficients,  relying on the linear results of Krylov \cite{Kr} and  the Leray-Schauder fixed point theorem (cf. \cite{Pl,Sf}).

Problems like \eqref{PD} arise  in  various applications such as for instance the heat transfer, mathematical modeling of semiconductor devices, stock market modelling, etc.
Many elliptic and parabolic equations with discontinuous coefficients are often proposed in models of the  deformation  of  composite materials, in the mechanics of membranes and films of simple non-homogeneous materials which form a linear laminated medium.  

In what follows we use the standard notation:

for any measurable   function $f\in L^p(Q)$ we indicate the norm by $\|f\|_{p,Q};$

denote by $D_iu$ the partial derivatives $\partial u/\partial x_i,$ $i=1,\ldots,n;$
$D f=(D_1f,\ldots,D_nf)$  is the gradient of $f$ and  $D^2f=\{D_{ij}f\}_{i,j=}^n$ is the Hessian matrix of $f;$

 the Sobolev space $W^{2,1}_p(Q)$  consists of functions heaving distributional derivatives in $x$ up to second order and one derivative in $t,$ endowed by the norm
 $$ 
 \|f\|_{W^{2,1}_p}(Q)=\|f\|_{p,Q}+\|Df\|_{p,Q}+\|D^2f\|_{p,Q}+\|u_t\|_{p,Q};
 $$ 

the letter  $C$ indicates a positive constant which value  varies from one appearance to another  and the standard summation on the repeated lower and upper indexes is adopted.

\section{Definitions and main assumptions}

Let $\mathcal{B}_r(x)=\{y\in \R^n: \ |x-y|<r\}$ be a ball in $\R^n$ and   $I_r(x,t)= \mathcal{B}_r(x)\times (t,t+r^2)$ be a \textit{parabolic cylinder} in $\R^n\times\R^+.$
For an arbitrary function $a\in L^1_\loc(\R^{n+1})$ we consider the \textit{mean oscillation} of $a$ with respect to $x$ in  $I_r(x,t)$ (see \cite{ANPS1,Kr})
$$
   \osc_x(a,I_r(x,t))= \Xint-_{t}^{t+r^2} \Xint-_{\mathcal{B}_r(x)} \Xint-_{\mathcal{B}_r(x)} |a(y ,\tau )- a(z ,\tau )|\,dy\,dz\,d\tau 
$$
and denote
$$
    a_R^{\#(x)}=   \sup_{(x,t)\in\R^{n+1}}  \,\sup_{r\leq R}\,\osc_x(a,I_r(x,t)).
$$
We say that $a\in VMO_x$ if $\lim_{R\to0}  a_R^{\#(x)} =0$. 

\begin{defin} A function $u(x,t)$ is a strong solution of \eqref{PD} if it is twice weakly differentiable in $x$ and once in $t$ with derivatives belonging to $L^{n+1}(Q)$, that is 
$u\in W^{2,1}_{n+1}(Q)\cap C(\overline{Q}),$ $u\vert_{\partial Q}=0,$ and satisfies $ \mathcal{P}u=f$ almost everywhere  in $Q$.
\end{defin}
For our convenience we suppose that the data  $a^{ij}$ and $f$ are extended as $0$ for $(x,t)\not\in \overline{Q}$. This permits us to extend the assumptions on  the coefficients in the  whole space. Let $Q_r(x,t)=I_r(x,t)\cap Q$ with $(x,t)\in Q$ and $\Omega_r=\mathcal{B}_r(x)\cap \Omega$.
We study the strong solvability of \eqref{PD}, imposing the following conditions on the data.

\begin{enumerate}
    \item $a^{ij}(x,t,u)$ and $f(x,t,u,p)$ are \textit{Carath\'eodory's functions}, that is, these are measurable in $(x,t)\in Q$ for all $(u,p)\in \R\times \R^n$ and continuous with respect to the other variables for a.a.  $(x,t)\in Q;$

    \item \textit{Strict parabolicity} of the operator $\mathcal{P}$: there exists a positive and non-increasing function $\lambda$ such that
    $$
    a^{ij}(x,t,u) \xi^i\xi^j\geq \lambda(|u|)|\xi|^2 \quad \text{a.e. in } Q, \ \forall\, \xi\in\R^n,
    $$
    and the matrix $\A=\{a^{ij}\}$ is symmetric;
    
    \item $a^{ij}(\cdot,u)\in VMO_x(Q)\cap L^\infty(Q)$  locally uniformly in $u\in[-M,M]$  for all $ M>0,$ that is, setting
    \begin{equation*}
    \begin{split}
         a_R^{ij\#(x)}:=\sup_{|u|\leq M} \sup_{r\leq R}\sup_{(x,t)\in Q}
    \Xint-_{Q_r(x,t)} \bigg(\Xint-_{\Omega_r(x)}|a^{ij}(y,t,u) \\
   -a^{ij}(z,t,u)|\,dz\bigg)dy\,dt,
    \end{split}
   \end{equation*}
    $$
    \gamma_M^{\#(x)}(R):= \max_{1\leq i,j\leq n} a_R^{ij\#(x)},
    $$
    we require $\lim_{R\rightarrow{0}} \gamma_M(R)\rightarrow 0;$ 
    
    \item \textit{Local uniform continuity} of $a^{ij}(x,t,u)$ with respect to $u$, uniformly in $(x,t)$:
    $$
    |a^{ij}(x,t,u)-a^{ij}(x,t,\overline{u})|\leq  \mu_M(|u-\overline{u}|) \quad \text{a.e. in $Q$} 
    $$
    for any $u,\overline{u}\in[-M;M]$. The function $\mu_M(\cdot)$ is non-decreasing and such that $\mu_M(\delta)\to 0$ for $\delta\rightarrow 0$. In addition $a^{ij}(x,t,0)\in L^\infty(Q);$ 
    
    \item \textit{Quadratic gradient growth}: 
    $$
    |f(x,t,u,p)|\leq \nu(|u|)(f_1(x,t)+|p|^2) \qquad \text{a.e. in $Q$}
    $$
    where $f_1\in L^{n+1}(Q)$ and $\nu(\cdot)$ is a positive and non-decreasing function;

    \item \textit{Sign condition} on $f$ with respect to $u$:
    $$
    \dfrac{\sign u\, f(x,t,u,p)}{\lambda(|u|)}\leq \nu_1(x,t)|p| +\nu_2(x,t) \quad \text{a.e. in Q}
    $$
    where $|u|>0$ and $\nu_1,\nu_2\in L^{n+1}(Q)$ are non-negative functions.
    \end{enumerate}

We are in a position now to announce  our main results. 
\begin{thm}[Existence]\label{thm1}
    Suppose that the conditions $(1)-(6)$ are fulfilled.
    Then the problem \eqref{PD} has a strong solution $ u\in W^{2,1}_{n+1}(Q)\cap C(\overline{Q})$.
\end{thm}

\begin{thm}[Uniqueness]\label{thm2}
Let the coefficients $a^{ij}$ of  $\mathcal{P}$ be independent of $u$ and the conditions $(1)-(6)$ be fulfilled.  Suppose, in addition, that $f(x,t,u,p)$  is non-increasing with respect to $u$ and Lipschitz continuous in $p,$ that is,  for all $p,p'\in \R^n$ it satisfies
\begin{equation}\label{uniq}
|f(x,t,u,p)-f(x,t,u,p')|\leq f_2(x,t,u)|p-p'| \quad a.e. \text{ in }Q,
\end{equation}
with $\sup_{|u|\leq M} f_2(x,t,u)\in L^{n+1}(Q)$ for any constant $M>0$. 

Then the problem \eqref{PD} cannot have more than one solution.
\end{thm}

\section{Auxiliary Lemmata}

The first step in our considerations is to obtain  boundedness of the solution, that  is essential for the proving of  existence . For this goal  we use the maximum principle for nonlinear parabolic operators obtained by Krylov in \cite{Kr2} and developed further by Nazarov and Ural'tseva in \cite{N} (see also \cite{T}.)
\begin{lem}\label{LM3.1}
Let $u\in W^{2,1}_{n+1}(Q) \cap C(\overline{Q})$, be a solution of  \eqref{PD} under the  assumptions (1),(2) and (6).  Then $u$ is essentially  bounded and
\begin{equation}\label{eq3.1}
\|u\|_{\infty,Q}\leq  C(n)d^{\frac{n}{n+1}}\Big( 1+ \frac{1}{d}\left\| \nu_1\right\|^{n+1}_{n+1,Q}\Big)\| \nu_2\|_{{n+1},Q}
\end{equation}
where $d=\diam\Omega$.    
\end{lem}
\begin{proof}
    Set $Q^+=\{(x,t)\in Q\colon\ u(x,t)>0\}$ for the subset  of  $Q$ where the solution $u$ is positive. In order to apply the maximum principle \cite{N,T} we need to linearize the operator $\mathcal{P}$.
    Dividing the both sides of the equation by $\lambda(|u|)$ and applying  the conditions (2) and (6)
    we obtain
  \begin{align*}
       \frac{u_t}{\lambda(|u|)} -\frac{a^{ij}(x,t,u)D_{ij}u}{\lambda(|u|)} &\leq |D u|\nu_1(x,t)+\nu_2(x,t)\\
      &\leq  \nu_1(x,t) \sign D_iu\cdot  D_iu  + \nu_2(x,t)
      \end{align*}
where we have used  that $|Du|\leq \sum_{i=1}^n|D_iu|=\sum_{i=1}^n \sign D_iu\cdot  D_iu.$

         Consider now the linear operator
         $$
    \overline{\mathcal{P}} u:=
     \sigma u_t -\overline{a}^{ij}(x,t)D_{ij}u -\overline{b}^i(x,t)D_i u\leq \nu_2(x,t)
    $$
where  $ \sigma=1/\lambda(|u|),$
$\overline{a}^{ij}=a^{ij}(x,t,u)/\lambda(|u|),$ and $\overline{b}^i=\nu_1 \cdot\sign D_iu \in L^{n+1}(Q)$. It  is  strictly parabolic   and denoting   $\overline{\A} = 
\{\overline{a}^{ij}\}$ we have    $\det \overline{\A}\geq 1$ that implies  $Tr\, \overline{\A}+\sigma>0.$ Direct application of the  Maximum principle    \cite{N}  gives
    $$
        \sup_{Q^+} u 
        \leq C(n)d^{\frac{n}{n+1}}\bigg( 1+ \frac{1}{d}\| \nu_1\|^{n+1}_{n+1,Q^+}\bigg)\| \nu_2\|_{{n+1},Q^+}.
    $$
        Analogously,  considering the subset $Q^-\subset Q$ of points in which $u(x,t)<0,$  we can estimate $\sup_{Q^-} (-u)$ in the same way. Unifying the both estimates we obtain \eqref{eq3.1}.
\end{proof}

The interior and boundary Harnack inequalities (cf. \cite{KrS,LU1}) imply Hölder continuity of the solution of \eqref{PD} and give an a priori estimate for it. This estimate allows us to  control the modulus of continuity $\omega(r)$ of $u$ in $Q$ that is, 
\begin{equation}\label{mdc}
 |\omega(r)|\leq C r^\alpha    
\end{equation}
where the constants $C$ and $\alpha$ depend on known quantities  but not on $u$.

The assumption (3) implies discontinuity of the coefficients $a^{ij}$ in $x$ and $t.$
In order to describe the discontinuity of the composition 
$a^{ij}(\cdot,\cdot,u(\cdot,\cdot)),$ 
we need to estimate its $VMO_x$ modulus.  For this goal we use the ideas developed  in \cite{Pl} (see also \cite{MPS,Sf}).
\begin{lem}\label{LM3.2}
Let $(1),(3),(4)$ hold and $u\in C^0(\overline{Q})$. Then 
$$
a^{ij}(x,t,u(x,t))\in VMO_x \cap L^\infty(Q)
$$
with $VMO$ modulus $\eta(r)$ bounded in terms of $\|u\|_{\infty,Q}$ and the modulus of continuity $\omega(r)$ of $u$.
\end{lem}
\begin{proof}
Let $(x_0,t_0)\in Q,$  denote by $M_u=\|u\|_{\infty,Q}$ and by $(3)$ we have 
\begin{align*}
    J^{ij}:=&\, \Xint-_{Q_r(x_0,t_0)} 
    \Xint-_{\Omega_r(x_0)}\big|a^{ij}(y,t,u(y,t))-a^{ij}(z,t,u(z,t))\,dz\big|\,dydt\\[6pt]
    \leq&\, \Xint-_{Q_r(x_0,t_0)} \Xint-_{\Omega_r(x_0)} \big|a^{ij}(y,t,u(y,t))-a^{ij}(y,t,u(x_0,t))\big|\\[6pt]
    +&\, \big|a^{ij}(y,t,u(x_0,t))-a^{ij}(z,t,u(x_0,t))\big|\\[6pt]
     +&\, \big|a^{ij}(z,t,u(z,t))-a^{ij}(z,t,u(x_0,t)) \big|\,dzdydt:=J_1+J_2+J_3,
\end{align*}    
after adding and subtracting suitable terms.
We note that $J_1$ does not depend on $z$ and by \eqref{mdc} we get
\begin{align*}
J_1=& \ \Xint-_{Q_r(x_0,t_0)}|a^{ij}(y,t,u(x,t))-a^{ij}(y,t,u(x_0,t))\big|\,dydt\\[6pt]
\leq& \  \Xint-_{Q_r(x_0,t_0)} \mu_M\big(|u(y,t)-u(x_0,t)|\big)\,dydt \leq \mu_M(\omega(r))
\end{align*}
where we have used $(4).$   
Analogously, $J_3$ does not depend on $y$ and we can estimate it in the same way as $J_1$, i.e., 
$
J_3\leq\mu_M(\omega(r))
$
while $J_2$ can be estimated via (3) obtaining  $ J_2\leq a^{ij\#(x)}_R. $
Unifying all these estimates we obtain
$$
  \eta(R)= \max_{1\leq i,j\leq n}\sup_{|u|\leq M_u} \sup_{(x_0,t_0)\in Q_r} \sup_{r\leq R} J^{ij}
   \leq \mu_M(\omega(R)) + a^{ij\#(x)}_R,
$$
where  the right-hand side tends to $0$ as $R\to 0$.
\end{proof}

The following a priori estimate is essential and permit us to determine  the space in which we can apply  the  Leray-Schauder fixed point theorem  (see \cite{MPS,Pl,Sf} and the references therein). 

\begin{lem}\label{LM3.3}
    Assume the condition (1)-(6) to be fulfilled. Then there exists a constant $C$ independent of $u$ such that 
    \begin{equation}\label{3.3}
        \|Du\|_{2(n+1),Q}\leq C
    \end{equation}  
for any solution $ u\in W^{2,1}_{n+1}(Q)\cap C(\ol Q)$  of the problem \eqref{PD}.
\end{lem}

\begin{proof}
Making use of the condition (5) and Lemmata~\ref{LM3.1} and~\ref{LM3.2}, we transform the equation $\mathcal{P} u=f$ into
$$
u_t-A^{ij}(x,t)D_{ij}u+B(x,t)|Du|^2+f_1(x,t)u=G(x,t)
$$
where
\begin{align}\label{eq 3.4} 
    \begin{split}
    &A^{ij}(x,t)=a^{ij}(x,t,u(x,t)), \qquad A^{ij}\in VMO_x\cap L^\infty(Q)\\[5pt]
    &B(x,t)=-\frac{f(x,t,u,Du)}{f_1(x,t)+|Du|^2}, \qquad \|B\|_{\infty,Q}\leq \nu(M_u)<\infty,\\[5pt]
    &G(x,t)=f_1(x,t)u+ \frac{f(x,t,u,Du)}{f_1(x,t)+|Du|^2}\,f_1(x,t),\\[5pt]
    & \|G\|_{n+1,Q}\leq \|f_1\|_{n+1,Q}\big(M_u+\nu(M_u)\big).
    \end{split}
\end{align}

Fixing  $u\in W^{2,1}_{n+1}(Q)\cap C(\ol{Q})$ at the data of the operator,    we  can  consider the one-parameter family  of problems  
\begin{equation}\tag{$\mathcal{F}_\sigma$}\label{3.5}
    \begin{cases}
    D_tv_\sigma - A^{ij}(x,t)D_{ij}v_\sigma +B(x,t)|D v_\sigma|^2&\\
    \phantom{D_tv_\sigma} +f_1(x,t)v_\sigma=\sigma G(x,t) & \text{ a.e. in }Q,\\[5pt]
    v_\sigma(x,0)=0  & \text{ on }\partial_pQ
\end{cases}
\end{equation}
with $\sigma\in[0,1].$ 
 Suppose that there exists a solution $v_\sigma\in W^{2,1}_{n+1}(Q)\cap C(\ol{Q})$  of \eqref{3.5} and write formally   $v_\sigma=\mathcal{F}(\sigma)$ to indicate it.

Then the question of the  uniqueness of that  solution   for any  $\sigma\in[0,1]$ arises naturally. 
\begin{prp}\label{prp1}
Let $0\leq \sigma_1<\sigma_2\leq 1$, $v_1=\mathcal{F}(\sigma_1),$ and $v_2=\mathcal{F}(\sigma_2)$. Then
\begin{equation}\label{3.6}
    \|v_2-v_1\|_{\infty,Q}\leq (\sigma_2 -\sigma_1)(M_u+\nu(M_u)).
\end{equation}
\end{prp}
\begin{proof}
%The functions $v_1$ and $v_2$ are solutions of \eqref{3.5}, that is
%$$
%\begin{cases} 
%  \mathcal{P} v_1\equiv D_tv_1-A^{ij}(x,t)D_{ij}v_1+B(x,t)|Dv_1|^2&\\
 % \phantom{ P v_1\equiv D_tv_1  }+f_1(x,t)v_1=\sigma_1G(x,t) & \text{ a.e. in } Q\\[5pt]
 % \mathcal{B} v_1\equiv v_1(x,0)=0 & \text{ on }\partial_pQ,  
%\end{cases}
%$$
%and
%$$
%\begin{cases}
 % \mathcal{P} v_2\equiv D_tv_2-A^{ij}(x,t)D_{ij}v_2+B(x,t)|Dv_2|^2&\\
 %\phantom{ \mathcal{P} v_2\equiv D_tv_2} +f_1(x,t)v_2=\sigma_2 G(x,t) &  \text{ a.e. in } %Q\\[5pt]
 % \mathcal{B} v_2\equiv v_2(x,0)=0 & \text{ on }\partial_pQ.
%\end{cases}
%$$

Let   $w=v_2-v_1=\mathcal{F}(\sigma_2) - \mathcal{F}(\sigma_1),$  that is 
$$
\begin{cases}
     w_t-A^{ij}(x,t)D_{ij}w+B(x,t)(|Dv_2|^2-|Dv_1|^2)&\\
  \qquad +f_1(x,t)w=(\sigma_2-\sigma_1)G(x,t) &   \text{ a.e. in } Q\\[5pt]
    w(x,0)=0 &\text{ on }\partial_pQ.    
\end{cases}
$$
 Further, we  linearize the equation  transforming the nonlinear term as follows 
 \begin{equation}\label{eqB}
\begin{split}
B(x,t)(|Dv_2|^2-|Dv_1|^2)&= B(x,t)   \sum_{i=1}^n (D_iv_2+D_iv_1) D_iw\\        % B(x,t) \sum_{i=1}^n \big[D_iw(D_iw+2D_iv_1)\big]\\
%&=\, 2B(x,t)\sum_{i=1}^n\int_0^1[\theta D_iw+D_iv_1]\,d\theta \cdot D_iw\\
&=: B^i(x,t) D_iw
\end{split}
\end{equation}
where   $B^i(x,t)= B(x,t)D_i(v_2+v_1)\in L^{n+1}(Q). $
Then we consider the linear problem
$$
\begin{cases} 
  \mathcal{P}_1w:=   w_t-A^{ij}(x,t)D_{ij}w+B^i(x,t)D_iw&\\
\phantom{ \mathcal{P}_1w:=  }\qquad +f_1(x,t)w=(\sigma_2-\sigma_1)G(x,t) &  \text{a.e. in } Q\\[5pt]
  \phantom{ \mathcal{P}_1w:=  }   w(x,0)=0 & \text{ on }\partial_pQ.    
\end{cases}
$$
By \eqref{eq 3.4} we deduce that 
$$
G(x,t)\leq f_1(x,t)(M_u+\nu(M_u)).
$$
Putting $M=(\sigma_2-\sigma_1)(M_u+\nu(M_u))$ we  obtain immediately that  
\begin{equation}\label{eq-2-12}
\begin{cases}    \mathcal{P}_1(w-M)\leq 0 &\quad \text{ a.e. in } Q\\[5pt]
    w-M\leq 0 &\quad  \text{  on } \partial_pQ
\end{cases}
\end{equation}
and hence the   maximum principle (cf. \cite{N})  gives $w\leq M$ a.e.  in $Q$ by the   maximum principle  \cite{N}. 

Taking  now  $\overline{w}=v_1-v_2=\mathcal{F}(\sigma_1)-\mathcal{F}(\sigma_2)$ and repeating the above considerations we obtain 
$$
\begin{cases}
\mathcal{P}_1\overline{w}=-(\sigma_2-\sigma_1)G(x,t) 
\leq  (\sigma_2-\sigma_1)f_1(M_u+\nu(M_u)) =\mathcal{P}_1 M\\
\overline{w}- M\leq 0
\end{cases}
$$
and by   the maximum principle it follows that  $\overline{w}\leq M$ a.e.  in $Q$ and hence  \eqref{3.6}.
\end{proof}

An immediate consequence of the Proposition~\ref{prp1} is the uniqueness of the solution of  \eqref{3.5}, providing that it exists.  In fact, putting $\sigma_1=\sigma_2$ in  \eqref{3.6} we get immediately $v_1=v_2$. 
Moreover,  \eqref{3.6} gives an estimate for  the norm $\|v_\sigma\|_{\infty,Q}$ of  any solution  of the problem \eqref{3.5}. 
In fact, setting $\sigma_1=0$  in \eqref{3.6}, one gets immediately
\begin{equation}\label{3.7}
\|v_\sigma\|_{\infty,Q}\leq M_u + \nu (M_u) \qquad \forall \  \sigma\in(0,1].
\end{equation}
We can also notice that    $v_0=\mathcal{F}(0)=0$ and  $v_1=\mathcal{F}(1)=u.$

Our goal now is to estimate the norm of the gradient of the solution of \eqref{3.5}, providing that it exists.

Consider again two solutions $v_1=\mathcal{F}(\sigma_1)$ and $v_2=\mathcal{F}(\sigma_2)$ of \eqref{3.5}  such that $0\leq  \sigma_1<\sigma_2\leq 1$.  Their difference $w=v_2-v_1$ solves the linear problem
\begin{equation}\label{3.8}
\begin{cases}
    w_t-A^{ij}(x,t)D_{ij}w+f_1(x,t)w=F(x,t)  &  \text{ a.e in }Q\\[5pt]
    w(x,0)=0  & \text{ on }\partial_pQ,
\end{cases}
\end{equation}
where
$$
F(x,t)=(\sigma_2-\sigma_1)G(x,t)-B(x,t)\big(|Dv_2|^2-|Dv_1|^2\big)
$$
with a norm
\begin{align*}
\|F\|_{n+1,Q}\leq&\,(\sigma_2-\sigma_1)\|G\|_{n+1,Q}\\
+&\, \|B\|_{\infty,Q}\big(\|Dw\|_{2(n+1),Q}^2+2\|Dv_1\|_{2(n+1),Q}^2\big).
\end{align*}
    
Choosing $\sigma_2-\sigma_1=\tau$  small enough  and using \eqref{eq 3.4} and \eqref{3.6} we obtain
\begin{align*}
     \|F\|_{n+1,Q}&\leq \, %\, \tau \|G\|_{n+1,Q}
    %+\|B\|_{\infty,Q}\big(\|Dv_2\|_{2(n+1),Q}^2+\|Dv_1\|_{2(n+1),Q}^2\big)\\
    \tau \|f_1\|_{n+1,Q}\big(M_u+\nu(M_u)\big)\\
    &+\, \nu(M_u)\big(\|Dw\|_{2(n+1),Q}^2+2\|Dv_1\|_{2(n+1),Q}^2\big)\\
    &\leq \,C_1\big(1+\|Dw\|_{2(n+1),Q}^2+2\|Dv_1\|_{2(n+1),Q}^2\big),
\end{align*}
where $C_1:=\max\{ \tau \|f_1\|_{n+1,Q}(M_u+\nu(M_u)), \nu(M_u)\}$. As it was already shown, the coefficients $A^{ij}(x,t)$ belong to $VMO_x\cap L^\infty(Q)$ and therefore the problem \eqref{3.8} has a unique solution which satisfies the a priori estimate
\begin{align*}
    \|w\|_{W^{2,1}_{n+1}(Q)}\leq C\|F\|_{n+1,Q}\leq C_1(1+\|Dw\|^2_{2(n+1),Q}+2\|Dv_1\|^2_{2(n+1),Q})
\end{align*}
with a  constant $C_1$ depending  on the data of the problem but not on the solution. In order to approximate the norm of the gradient we adopt the  Solonnikov inequality that is an anisotropic version of the  Gagliardo-Nirenberg  interpolation inequality  (cf. \cite{Ni,Sl}) to the parabolic case, obtaining

%Suppose that the function $w(x,t)$ is extended as zero in $\R^{n+1}\setminus Q$. For a.a. $t\in \R$,  we have that 
 %$w(\cdot,t)\in L^\infty(\R^n)$, $D^2w(\cdot,t)\in L^{n+1}(\R^n)$, then  the Gagliardo-Nirenberg inequality gives
%\begin{equation}\label{G-N1}
% \|Dw(\cdot,t)\|^2_{2(n+1)}\leq K^2\|D^2 w(\cdot,t)\|_{n+1}\|w(\cdot,t)\|_{\infty}.  
%\end{equation}
%where $p$ is given by
%\begin{equation}\label{coeffGN}
%\frac{1}{p}=\frac{1}{n}+\alpha\bigg(\frac{1}{n+1}-\frac{2}{n}\bigg).  
%\end{equation}
%Integrating \eqref{G-N1} with respect to $t$ and keeping in mind that the function is equal to zero away from the cylinder, we can write
%\begin{align*}
%\int_0^T\int_\Omega|Dw(x,t)|^{2(n+1)}\,dxdt\leq  K^p\int_0^T&\bigg(\int_\Omega |D^2w(x,t)|^{n+1}\,dx\bigg)^{\alpha p}\|w(\cdot,t)\|^{(1-\alpha)p}_{\infty,\Omega}\,dt.
%\end{align*}
%Direct use of the Hölder inequality gives  the desired estimate
%\begin{align}\label{G-N-H}
%    \|Dw\|^2_{p,Q}\leq C \|D^2w\|^{\alpha p}_{n+1,Q}\cdot \|w\|^{(1-\alpha)p}_{\infty,Q}
%\end{align}
%where $C=C(K,T,n,\alpha,p)$. Taking $\alpha=1/2$  and making use of \eqref{3.7} we obtain
\begin{align*}
    \|Dw\|^2_{2(n+1),Q}&\leq C\|w\|_{\infty,Q}\|D^2w\|_{n+1,Q}\\
    &\leq C\tau \big(M_u +\nu(M_u)\big)\|D^2w\|_{n+1,Q}\\
    &\leq C_1\tau\|w\|_{W^{2,1}_{n+1}(Q)}
    \leq C_1\tau\big(1+\|Dw\|^2_{2(n+1),Q}+\|Dv_1\|^2_{2(n+1),Q}\big)
\end{align*}
with a constant $C_2$ depending on known quantities.

Choosing $\tau$ smaller, if necessary, such that $C_2\tau<1,$ we get
$$
\|Dw\|^2_{2(n+1),Q}\leq C_3\big(1+\|Dv_1\|^2_{2(n+1),Q}\big).
$$
Since $v_2=w+v_1$, we deduce that
\begin{equation}
    \begin{split}\label{3.9}
   \|Dv_2\|^2_{2(n+1),Q}&\leq \|Dw\|^2_{2(n+1),Q}+\|Dv_1\|^2_{2(n+1),Q} \\
    &\leq C_4(1+ \|Dv_1\|^2_{2(n+1),Q}).
    \end{split}
\end{equation}
Taking $\sigma_1=0$ and $\sigma_2=\tau,$ that is $v_1=0$ and  $v_2=v_\tau$  we obtain the a  priori estimate 
\begin{equation}\label{eq3.10}
\|Dv_\tau\|^2_{2(n+1),Q}\leq C_4.
\end{equation}

To prove the solvability of \eqref{3.5} under the assumption $\sigma=\tau$, we consider the operator
$$
\mathfrak{F}: \mathcal{S}(Q)\longrightarrow W^{2,1}_{n+1}(Q)\cap C(\ol{Q})
$$
which associates to  any $z\in\mathcal{S}(Q)=\{z\in C^0(\overline{Q}):Dz\in L^{2(n+1)}(Q)\}$ the  unique solution $v\in W^{2,1}_{n+1}(Q)\cap C(\ol{Q})$ of the linear problem
\begin{equation}\tag{$\mathfrak{F}$}\label{eq-farcF}
    \begin{cases}
v_t-A^{ij}(x,t)D_{ij}v=\tau G(x,t)-B(x,t)|Dz|^2-f_1(x,t)z & \text{ a.e in }Q,\\
v(x,0)=0  &  \text{ on } \partial_pQ,
\end{cases}
\end{equation}
and we adopt the notion $\mathfrak{F}(z)=v$ to indicate this solution.
The existence of the last one follows by \eqref{eq 3.4} and \cite{Kr}.

The Sobolev embedding theorem implies a compact embedding of the space $W^{2,1}_{n+1}(Q)\cap C(\ol{Q})$ into $\mathcal{S}(Q)$ that gives also compactness of the operator $\mathfrak{F}$ considered as a mapping from  the Banach space  $\mathcal{S}(Q)$ into itself.
Let us  consider now the one parameter family of problems $\mathfrak{F}_\rho :\mathcal{S}(Q)\times[0,1]\to \mathcal{S}(Q)$ defined as  
\begin{equation}\tag{$\mathfrak{F}_\rho$}
    \begin{cases}
v_t-A^{ij}(x,t)D_{ij}v&\\
\phantom{v_t}=\rho(\tau G(x,t)-B(x,t)|Dv|^2-f_1(x,t)v)& \text{ a.e in }Q,\\[5pt]
v(x,0)=0  &  \text{ on } \partial_pQ.
\end{cases}
\end{equation}
If we can show that for any $\rho\in[0,1]$ the norm of the  solution of $\mathfrak{F}_\rho(v)=v$ can be estimated via a constant independent of $\rho$ than  we will  by able to apply   the Leray-Schauder theorem  that ensures the existence of a fixed point of the problem $\mathfrak{F}_1(v)=v,$  which is exactly the solution  of \eqref{3.5} with $\sigma=\tau.$

In fact, the desired estimate follows as    \eqref{3.7} and \eqref{eq3.10}, that is   $\|v\|_{\mathcal{S}(Q)}\leq C$ with a constant independent of $v$ and $\rho\in[0,1]$.

To conclude the proof of Lemma \ref{LM3.3}, we cover the interval $[0,1]$ by a finite number of subintervals each one of length $\delta\leq\tau$. Then, iterating the above procedure and employing \eqref{3.9}, we estimate $\|Dv_{k+1}\|_{2(n+1),Q}$ in terms of $\|Dv_k\|_{2(n+1),Q}$ for $k=0,1,\dots,m-1$ with $m\delta=1$. Since \eqref{3.5} with $\sigma=1$ coincides with the original problem \eqref{PD}, having in mind the uniqueness of the solution to \eqref{3.5} already proved, we obtain the desired estimate \eqref{3.3}.
\end{proof}

\section{Main results}

\begin{proof}[Theorem~\ref{thm1}~(Existence)] To prove  existence of solution of \eqref{PD}, we use one more time the Leray-Schauder theorem. Fixing an arbitrary function $v\in\mathcal{S}(Q)$ in the data $a^{ij}(x,t,v)$ and $f(x,t,v,Dv)$ of \eqref{PD}, we linearize the problem obtaining
\begin{equation}\tag{$\mathfrak{F}_2 $}
     \begin{cases}
 D_tu-a^{ij}(x,t,v)D_{ij}u=f(x,t,v,Dv)  \quad &\text{a.e in }Q,\\
    u(x,0)=0 &\text{on } \partial_p Q.
    \end{cases}
\end{equation}
Then me obtain a  linear strictly parabolic equation with bounded $VMO_x$ principal coefficients and right-hand side belonging to $L^{n+1}(Q)$. Therefore, by \cite{Kr}, there exists a  unique strong solution $u\in W^{2,1}_{n+1}(Q)\cap C(\ol{Q})$ and we indicate it  with $u=\mathfrak{F}_2(v)$.

The compactness of $\mathfrak{F}_2 $ acting from $\mathcal{S}(Q)$  into itself, does follow in the same way as the compactness of $\mathfrak{F} $ above.

To prove the continuity of $\mathfrak{F}_2$ we take a sequence 
$\{v_h\}_{h=1}^\infty\in \mathcal{S}(Q)$ such that $\|v_h-v\|_{\mathcal{S}(Q)}\to 0$  and set $u_h=\mathfrak{F}_2(v_h)$. Then we consider the difference  $u_h-u$ which is a solution of $\mathfrak{F}_2(v_h)- \mathfrak{F}_2(v)$
\begin{align*}
\begin{cases}
  D_t(u_h-u)-a^{ij}(x,t,v_h)D_{ij}(u_h-u)&\\
  \phantom{D_t(u_h-u)}=\big(a^{ij}(x,t,v_h)-a^{ij}(x,t,v)\big)D_{ij}u&\\
  \phantom{D_t(u_h-u)}+f(x,t,v_h,Dv_h)-f(x,t,v,Dv) &\text{a.e in }Q\,,\\[6pt]
  u_h(x,0)-u(x,0)=0  &\text{on } \partial_p Q\,.
\end{cases}
\end{align*}
Further,  \cite[Theorem 2.1]{Kr} yields an a priori estimate of the solution
\begin{equation}\label{N-O}
\begin{split}
       \|u_h-u\|_{W^{2,1}_{n+1}(Q)}&  \leq  C \Big[\|(a^{ij}(\cdot,v_h)-a^{ij}(\cdot,v)) D_{ij}u\|_{n+1,Q}\\
&+\|f(\cdot,v_h,Dv_h)-f(\cdot,v,Dv)\|_{n+1,Q}\Big]\\
&\leq \  C 
\Big[\|\mu(v_h-v) D_{ij}u\|_{n+1,Q}\\
  & +\|f(\cdot,v_h,Dv_h)-f(\cdot,v,Dv)\|_{n+1,Q}\Big]
\end{split}
\end{equation}
 Condition (4) ensures the convergence of the first  term while condition  (5) ensures the  continuity of the Nemytskii operator
$$
%\mathcal{S}(Q)\ni  v\mapsto a^{ij}(x,t,v)\in  L^\infty(Q), \qquad 
 \mathcal{S}(Q)\ni  v\mapsto f(x,t,v,Dv)\in L^{n+1}(Q).
$$ 
 Hence, the right-hand side in \eqref{N-O} tends to zero when $h\rightarrow\infty$ which gives continuity of $\mathfrak{F}_2$.
Further on, the estimates  \eqref{eq3.1} and \eqref{3.3} ensure $\|u\|_{\mathcal{S}(Q)}\leq C$ with a constant independent of $u$ and $\rho\in[0,1]$. This ensures existence of a fixed point $u\in \mathcal{S}(Q)$ of  the operator $\mathfrak{F}_2,$ which belongs to  $ W^{2,1}_{n+1}(Q)\cap C(\ol{Q})$ by the definition of $\mathfrak{F}_2,$ and this implies existence of   solution to \eqref{PD}.
\end{proof}

\begin{proof}[Theorem~\ref{thm2}~(Uniqueness)]
We argue by contradiction. 
Supposing that  $u,v\in W^{2,1}_{n+1}(Q)\cap C(\ol{Q})$ are two solutions of \eqref{PD}, then
$$
\begin{cases}
     D_t(u-v)-a^{ij}(x,t)D_{ij}(u-v)=f(x,t,u,Du)- f(x,t,v,Dv)  &\text{ a.e in }Q\,,\\
  u-v=0  & \text{ on } \partial_pQ\,.
\end{cases}
$$
Let  $Q^+=\{(x,t)\in Q: u(x,t)>v(x,t)\}$. Since $f(x,t,z,p)$ is non-increasing in $z$, it follows that $f(x,t,u,Du)-f(x,t,v,Du)\leq 0$    for a.a. $(x,t)\in Q^+.$ Later on, using the condition \eqref{uniq}  and denoting  $M_v=\|v\|_{\infty,Q}$, we can rewrite the equation above as follows
\begin{align*}
 D_t(u-v)&-a^{ij}(x,t)D_{ij}(u-v)=f(x,t,u,Du)-f(x,t,v,Du)\\
 & +  f(x,t,v,Du)-f(x,t,v,Dv)\\
 &\leq  \sup_{|v|\leq M_v}f_2(x,t,v)|Du-Dv|\\
 & = \sup_{|v|\leq M_v}f_2(x,t,v) \sum_{i=1}^n (\sign D_i(u-v))D_i(u-v) .  
\end{align*}
Therefore taking $b^i=- \sup_{|v|\leq M_v}f_2(x,t,v) \cdot \sign D_i(u-v)$ we obtain the problem
$$
\begin{cases}
    D_t(u-v)-a^{ij}(x,t)D_{ij}(u-v) +b^iD_i(u-v)\leq 0 &\text{ a.e in }Q\,\\
    u-v=0 & \text{ on } \partial_p Q\,.
\end{cases}
$$
The maximum principle gives $\max_{Q^+}(u-v)\leq 0$. Analogously, taking the set  $Q^-=\{(x,t)\in Q:v>u\}$ we get $\max_{Q^-}(v-u)\leq 0$ which gives $u=v$ in $Q$ and thus the claim of Theorem \ref{thm2}.
\end{proof}

{\bf Acknowledgement.} {\small
The author is indebted to  the referee for the variable report that brings to the improvement of the paper.
The author is a member of INDAM-GNAMPA.}


\begin{thebibliography}{123}


\bibitem{AFS} E.A. Alfano, L. Fattorusso, L. Softova, 
Boundedness of the solutions of a kind of nonlinear parabolic systems,
{\it  J. Differ. Equ.}, {\bf 360}, \mbox{51--66} (2023).

  \bibitem{ANPS1}
 D.E. Apushkinskaya, A.I.  Nazarov, D.K. Palagachev, L.G. Softova, 
Nonstationary Venttsel problems with discontinuous data, {\it
J. Differ. Equ.}, {\bf  375}, \mbox{538--566} (2023).
 \bibitem{ANPS2}
 D.E. Apushkinskaya, A.I.  Nazarov, D.K. Palagachev, L.G. Softova, 
Nonstationary Venttsel problem with $VMO_x$ leading coefficients, {\it 
Dokl. Math.}, {\bf  107}, No.~2, \mbox{97--100} (2023); {\it Transl. Dokl. Ross. Akad. Nauk, Mat. Inform. Protsessy Upr.},  {\bf 510}, \mbox{13--17} (2023).

\bibitem{ANPS3}
 D.E. Apushkinskaya, A.I.  Nazarov, D.K. Palagachev, L.G. Softova, Venttsel boundary value problems with discontinuous data {\it SIAM J. Math. Anal.}, {\bf 53}, \mbox{221--252} (2021).

\bibitem{BPS}
S.-S. Byun, D.K. Palagachev, L.G. Softova, 
Survey on gradient estimates for nonlinear elliptic equations in various function spaces. {\it 
St. Petersbg. Math. J.}, {\bf 31}, No.~3, \mbox{401--419} (2020) {\it Algebra Anal.}, {\bf 31,} No.~3, \mbox{10--35} (2019).
\bibitem{BSf}
S.-S. Byun,  L.G. Softova, 
Asymptotically regular operators in generalized Morrey spaces, {\it 
Bull. Lond. Math. Soc.}, {\bf  52}, No.~1, \mbox{64--76} (2020).
\bibitem{CFL} F. Chiarenza,  M. Frasca,  P. Longo, $W^{2,p}$ solvability of the Dirichlet problem for nondivergence form elliptic equations with $VMO$ coefficients, {\it Trans. Amer. Math. Soc.}, {\bf 336}, \mbox{841--853} (1993).

\bibitem{DK} H. Dong,  D. Kim, On the {$L_p$}-solvability of higher order parabolic and
elliptic systems with $BMO$ coefficients, {\it Arch. Ration. Mech. Anal.}, {\bf 199}, \mbox{889--941} (2011).

 
 
\bibitem{FK} S. Fučík, A. Kufner, 
Nonlinear Differential Equations,
{\it Elsevier Sci. Publishers,} New York, 1980. 

\bibitem{GT} D. Gilbarg, N.S. Trudinger, Elliptic Partial Differential Equations of Second Order, 2nd Edition, {\it Springer,} Berlin,  1983.

\bibitem{KimKr} 
D. Kim, N.V.  Krylov, 
Parabolic equations with measurable coefficients, {\it 
Potential Anal.},  {\bf 26}, No.~4, \mbox{345--361} (2007).

\bibitem{KrS} N.V. Krylov, M.V. Safonov,
A property of the solutions of parabolic equations with measurable coefficients, {\it Math. USSR-Izvestiya}, {\bf 16},  No.~1, \mbox{151--164} (1981).

\bibitem{Kr} N.V. Krylov, 
 Parabolic and elliptic equations with $VMO$ coefficients, {\it Commun. Partial Differ. Equations}, {\bf 32},  No.~3, \mbox{453--475} (2007).

 \bibitem{Kr1} N.V. Krylov, 
Lectures on Elliptic and Parabolic Equations in Sobolev Spaces, {\it Amer. Math. Soc.}, Grad. Stud. Math., {\bf 96}, 2008.

\bibitem{Kr2}
N. V. Krylov, 
On the maximum principle for nonlinear parabolic and elliptic equations, {\it Mathem. of the USSR-Izvestiya},{\bf 13}, No.~2, \mbox{335--347} (1979).

\bibitem{LSU} O.A. Ladyženskaja, V.A. Solonnikov, N.N. Ural'tseva,
Linear and Quasilinear Equations of Parabolic Type, {\it Amer. Math. Soc.}, {\bf 23}, 1968.

\bibitem{LU} O.A. Ladyženskaja, N.N. Ural'tseva, Linear and Quasilinear Elliptic Equations,  {\it Academic Press,} New York, London, 1968.

\bibitem{LU1} O.A. Ladyženskaja, N.N. Ural'tseva, A survey of results on the solubility of boundary-value problems for second-order uniformly elliptic and parabolic quasi-linear equations having unbounded singularities,  {\it Russian Math. Surveys}, {\bf 41}, No.~5, \mbox{1--31} (1986).


\bibitem{MPS} A. Maugeri, D.K. Palagachev, L.G. Softova, Elliptic and Parabolic Equations with Discontinuous Coefficients. {\it Mathematical Research}, 109. {\it Wiley-VCH Verlag,} Berlin GmbH, Berlin, 2000. ISBN: 3-527-40135-0.

\bibitem{N} A.I. Nazarov,  N.N. Ural'tseva, Convex monotone hulls and estimation of the maximum of a solution of a parabolic equation, {\it Soviet Math.}, {\bf 37}, \mbox{851--859} (1987).

\bibitem{Ni} Nirenberg, L., On elliptic partial differential equations, {\it Ann. Scuola Norm. Sup. Pisa},  {\bf 13}, No~3, \mbox{115--162} (1959).


\bibitem{Pl} D.K. Palagachev, Quasilinear elliptic equations with $VMO$ coefficients, {\it Trans. Amer. Math. Soc.}, {\bf 347}, \mbox{2481--2493} (1995).

\bibitem{RPS}
D.K. Palagachev, L. Recke, L.G. Softova,
Applications of the differential calculus to nonlinear elliptic operators with discontinuous coefficients,
{\it 
Math. Ann.},  {\bf  336}, No.~3, \mbox{617--637} (2006).

\bibitem{RS}
L. Recke, L.G. Softova,
Nonlinear parabolic operators with perturbed coefficients, 
{\it Comm. Mathem.  Appl.}, {\bf 9}, No.~3, \mbox{277--292} (2018).


\bibitem{Sf} L.G. Softova, Quasilinear parabolic operators with discontinuous ingredients, {\it Nonlin. Anal., Theory Meth. Appl., Ser. A,} {\bf 52}, \mbox{1079--1093} (2003).

\bibitem{Sl} 
V.A. Solonnikov, 
Inequalities for functions of the classes $W^{\vec m}_p(\R^n),$  {\it 
J. Sov. Math.,}  {\bf 3},  \mbox{549--564} (1975).
 

   
\bibitem{T} K. Tso, On an Aleksandrov-Bakel'man type maximum principle for second-order parabolic equations, {\it Commun. Partial Differ. Equ.}, {\bf 10}, 543--553 (1985).
\end{thebibliography}
\end{document}